\documentclass{amsart}
\usepackage{amscd, amssymb, euscript}

\newcounter{noindnum}[subsection]
\newcommand{\noindstep}{\refstepcounter{noindnum}{\rm(}\alph{noindnum}\/{\rm)
}}
\renewcommand{\phi}{\varphi}
\renewcommand{\epsilon}{\varepsilon}

\newcommand{\Dmod}{\EuScript D}

\newcommand{\B}{\mathcal B}
\newcommand{\E}{\mathcal{E}}
\newcommand{\F}{\mathcal{F}}
\newcommand{\J}{\mathcal J}
\newcommand{\cL}{\mathcal L}
\newcommand{\U}{\mathcal{U}}
\newcommand{\V}{\mathcal{V}}
\newcommand{\X}{\mathcal X}

\newcommand{\C}{\mathbb C}
\newcommand{\Pro}{\mathbb P}
\newcommand{\Z}{\mathbb{Z}}

\newcommand{\smcn}{{\rm;}}
\newcommand{\vac}{|0\rangle}

\DeclareMathOperator{\Res}{Res}

\DeclareMathOperator{\Lie}{Lie}
\DeclareMathOperator{\ad}{ad}
\DeclareMathOperator{\Ad}{Ad}
\DeclareMathOperator{\Fun}{Fun}
\DeclareMathOperator{\rk}{rk}
\DeclareMathOperator{\Hom}{Hom}
\DeclareMathOperator{\End}{End}
\DeclareMathOperator{\Ind}{Ind}
\DeclareMathOperator{\tr}{tr}
\DeclareMathOperator{\Sym}{Sym}

\DeclareMathOperator{\gr}{gr}

\newcommand{\conn}{\mathop{\mathbf{Conn}}\nolimits}
\newcommand{\Conn}{\mathcal{C}onn_{N(0)+(\infty)}^\fh}
\newcommand{\ConnReg}{\underline{\mathcal{C}onn}_{N(0)+(\infty)}^\fh}
\newcommand{\Bun}{\mathcal{B}un_{N(0)+(\infty)}^\fh}
\newcommand{\MM}{\mathbb{M}_{N,k}}
\newcommand{\MMR}{\mathbb{M}_{N,k,reg}}
\newcommand{\fgp}{{\hat\fg_+}}
\newcommand{\pireg}[1]{{\hat\pi^{#1}_{N,reg}}}
\newcommand{\fhp}{{\hat\fh_+}}

\newcommand{\fh}{\mathfrak h}
\newcommand{\fg}{\mathfrak g}
\newcommand{\fu}{\mathfrak u}
\newcommand{\fb}{\mathfrak b}
\newcommand{\fsl}{\mathfrak{sl}}

\theoremstyle{plain}
\newtheorem{theorem}{Theorem}
\newtheorem{proposition}{Proposition}[section]
\newtheorem{lemma}{Lemma}[section]
\newtheorem*{corollary}{Corollary}

\theoremstyle{definition}

\theoremstyle{remark}
\newtheorem*{remark}{Remark}
\newtheorem*{remarks}{Remarks}

\author{Roman M. Fedorov}
\title[Irregular Wakimoto modules]{Irregular Wakimoto modules\\
and the Casimir connection}
\email{fedorov@bu.edu}
\address{Mathematics \& Statistics\\
Boston University\\
111 Cummington St\\
Boston, MA}

\begin{document}

\begin{abstract}
We study some non-highest weight modules over an affine
Kac--Moody algebra $\hat\fg$ at non-critical level. Roughly
speaking, these modules are non-commutative localizations of
some non-highest weight ``vacuum'' modules. Using free field
realization, we embed some rings of differential operators in
endomorphism rings of our modules.

These rings of differential operators act on a localization of
the space of coinvariants of any $\hat\fg$-module with respect
to a certain \emph{level subalgebra}. In a particular case this
action is identified with the Casimir connection.
\end{abstract}

\maketitle

\section{Introduction}
Let $\fg$ be a simple Lie algebra. Consider the affine
Kac--Moody algebra $\hat\fg$ that is a non-split 1-dimensional
central extension of the loop algebra $\fg\otimes\C((t))$. It
is well known that such an extension is unique up to an
isomorphism. Denote a generator of the center by $\mathbf{1}$.
Let us set
$\U_k\hat\fg=\U\hat\fg/((k-\mathbf{1})\cdot\U\hat\fg)$, where
$\U\hat\fg$ is the universal enveloping algebra.

Representations of Kac--Moody algebras have been studied for a
few decades, see for example~\cite{Kac}. However, the study was
mostly concerned with \emph{highest weight representations}. In
particular, these representations have the property that
$t\fg[[t]]$ acts locally nilpotently. One of the simplest (and
most important) examples of such a representation is the
so-called vacuum module
\[
    \U_k\hat\fg\otimes_{\U\fg[[t]]}\C=
    \U_k\hat\fg/(\U_k\hat\fg\cdot\fg[[t]]).
\]
It is well known that the ring of endomorphisms of this module
is equal to $\C$ for all but one value of $k$, this
exceptional $k$ is called the \emph{critical level},
see~\cite[Ch.~3]{FrenkelBook}.

In this paper we consider a more general class of
representations, called \emph{smooth} representations,
i.e., representations in which every vector is annihilated by
$t^N\fg[[t]]$ for large enough $N$. An example of such
representation is
\[
    \U_k\hat\fg\otimes_{\U(t^N\fg[[t]])}\C=
    \U_k\hat\fg/(\U_k\hat\fg\cdot t^N\fg[[t]]),
\]
where $N$ is a positive integer. We would like to calculate the
ring of endomorphisms, however we shall have to make a few
simplifications. Firstly, instead of $t^N\fg[[t]]$, we shall be
using $t^N(\fu_-\oplus\fu)\oplus t^{N+1}\fg[[t]]$, where
$\fg=\fu_-\oplus\fh\oplus\fu$ is a triangular decomposition of
$\fg$. Secondly, we shall consider a certain
\emph{non-commutative localization} of the corresponding vacuum
module, see~\S\ref{NONHIGHEST}. The endomorphisms of this
localization can be promoted to those of the original module by
``clearing denominators'', but we do not know how to
characterize those endomorphisms that have ``no denominators'',
see the corollary of Proposition~\ref{PrEnd}.

We shall provide the conjectural answer for the ring of
endomorphisms: this is essentially the ring of differential
operators on a certain complement of hyperplanes arrangement.
However, we shall only prove that the latter ring injects into
the former, see Theorem~\ref{ThEnd}.

We shall use this injection to construct a natural functor $\F$
from the category of representations of $\hat\fg$ to the
category of $\Dmod$-modules on the complement of hyperplanes.

\subsection{Irregular Wakimoto modules}
To construct the above injection we shall use free field
realization. We shall define some irregular analogues of
Wakimoto modules. Then we shall show that, after localization,
our modules are isomorphic to the corresponding irregular
Wakimoto modules. The sought-after endomorphisms are
transparent on the Wakimoto side. This is somewhat analogous to
the construction of central elements in the vertex algebra,
corresponding to $\fg$ on the critical level,
cf.~\cite{FrenkelBook}.

\subsection{Relation to the Casimir connection}
Let $V$ be a $\fg$-module, $\fh^r\subset\fg$ be the regular
part of a Cartan subalgebra; let us view $V\times\fh^r$ as a
trivial vector bundle on $\fh^r$. A certain connection on this
bundle has been constructed independently by De~Concini
(unpublished), Felder--Markov--Tarasov--Varchenko~\cite{FMTV,
TarasovVarchenko}, and Milson--Toledano
Laredo~\cite{MilsonToledanoLaredo}.
Following~\cite{ToledanoLaredo2}, we shall call it the Casimir
connection.

The main feature of this connection is that its monodromy gives
\emph{the quantum Weyl group action}, as was recently
proved by Toledano Laredo in~\cite{ToledanoLaredo2}.

It turns out that if one takes $N=1$ and restricts our functor
$\F$ to certain \linebreak $\hat\fg$-modules, induced from
$\fg$-modules, then $\F$ coincides with the Casimir connection up
to a~twist, see~\S\ref{CASIMIR}.

\subsection{Relation to results of D.~Ben~Zvi and E.~Frenkel}
The representation theory of $\hat\fg$ is closely related to
numerous structures on moduli spaces of principal bundles on
complex curves. In particular, using $\hat\fg$ and the Virasoro
algebra, E.~Frenkel and D.~Ben--Zvi have constructed
in~\cite{FrenkelBenZvi} the following deformation-degeneration
picture for a smooth projective curve:
\[
\begin{CD}
\mbox{KZB (or Heat) connection} @>>>
\mbox{\parbox{3cm}{Quadratic part of\\ Beilinson--Drinfeld\\
system (Opers)}}\\
@VVV @VVV\\
\mbox{Isomonodromic Deformation} @>>>
\mbox{\parbox{3cm}{Quadratic part
of\\Hitchin System}}
\end{CD}
\]
This also works in a ramified case. However, in the case when
the ramification divisor has multiple points (we shall call
this case \emph{irregular}), a new direction arises, which was
missing in~\cite{FrenkelBenZvi}. For isomonodromic deformation
this direction was put in the framework of~\cite{FrenkelBenZvi}
in~\cite{FedorovIsoStokes}.

According to the above picture, the quasi-classical limit of
KZB operators gives isomonodromic hamiltonians. As will be
briefly explained in~\S\ref{ISOMONODROMY}, the quasi-classical
limit of operators, coming from Theorem~\ref{ThEnd}, gives
hamiltonians for irregular direction of isomonodromy. Thus one
could think about these operators as about ``irregular
direction of KZB''. We hope to return to this point elsewhere.
Note also, that these irregular directions are in a sense
local, so that for every curve the picture has a global version
(cf.~\S\ref{GLOBAL} and~\S\ref{ISOMONODROMY}).

We would like to remark, that the irregular opers were studied
in~\cite{FeiginFrenkelLaredo}. We expect that the limits of our
operators as the level tends to the critical one, should be
related to irregular opers. Note also, that the quasi-classical
limit was identified with isomonodromic deformation
in~\cite{Boalch}.

\subsection{Acknowledgments}
The author wants to thank Dima Arinkin, David Ben--Zvi, Roman
Bez\-ru\-kav\-ni\-kov, Pasha Etingof, Edward Frenkel, Dennis
Gaitsgory, Victor Ginzburg, Valerio Toledano Laredo, Matthew
Szczesny, and Alexander Varchen\-ko for valuable discussions.
The author would like to thank the referee for valuable
comments. The idea to look at the action of endomorphisms on
the spaces of coinvariants belongs to Pasha Etingof. I have
also learnt recently that a similar work is done in a paper in
preparation by Boris Feigin, Edward Frenkel, and Valerio
Toledano Laredo.

\section{Main results}
\subsection{Notation}
We are working over the field $\C$ of complex numbers. Let $G$
be a connected simple Lie group, $\fg$ its Lie algebra. Fix a
triangular decomposition $\fg=\fu_-\oplus\fh\oplus\fu$. Let
$U_-$ and $U$ be the maximal unipotent subgroups of $G$ with
$\Lie(U_-)=\fu_-$, $\Lie(U)=\fu$, let $B_-$ and $B$ be Borel
subgroups with $\Lie(B_-)=\fb_-:=\fh\oplus\fu_-$,
$\Lie(B)=\fb:=\fh\oplus\fu$. We denote the unit element in a
group by 1.

Let $\Delta\subset\fh^*$ be the root system,
$\Delta_+\subset\Delta$ be the set of positive roots. Denote by
$\alpha_1,\ldots,\alpha_{\rk\fg}$ the simple roots. For each
$\alpha\in\Delta_+$ fix an $\fsl_2$-triple
$(e_\alpha,f_\alpha,h_\alpha)$. We denote $e_{\alpha_i}$ by
$e_i$ and define $f_i$ and $h_i$ similarly.

Denote by $\fh^r$ the regular part of $\fh$, namely,
\[
    \fh^r:=\{h\in\fh:\forall\alpha\in\Delta\;\alpha(h)\ne0\}.
\]
Similarly, $\fh^{*,r}$ is the regular part of the dual space
\[
    \fh^{*,r}:=\{\chi\in\fh^*:{\forall\alpha\in\Delta}\;
    \chi(h_\alpha)\ne0\}.
\]
We define the invariant bilinear form on $\fg$ by
\[
(x,y)=\frac1{2h^\vee}\tr(\ad_x\ad_y),
\]
where $x,y\in\fg$, $h^\vee$ is the dual Coxeter number of
$\fg$, $\ad$ is the adjoint representation of $\fg$.

\subsubsection{The affine Kac--Moody algebra}
Let $\fg((t))=\fg\otimes\C((t))$ be the formal loop algebra of
$\fg$. Let $\hat\fg=\fg((t))\oplus\C\mathbf{1}$ be the affine
Kac--Moody algebra of $\fg$, here $\mathbf{1}$ is the standard
generator of the center. The Lie algebra structure is given by
\begin{equation}\label{affinekacmoody}
    [x_1\otimes g_1(t),x_2\otimes g_2(t)]=
    [x_1,x_2]\otimes(g_1(t)g_2(t))+
    (x_1,x_2)\Res_t(g_1dg_2)\mathbf{1}.
\end{equation}
We denote $e_\alpha\otimes t^n\in\hat\fg$ by $e_{\alpha,n}$ and
define $f_{\alpha,n}$, $h_{\alpha,n}$ similarly. We also set
$e_{i,n}:=e_{\alpha_i,n}$, $f_{i,n}:=f_{\alpha_i,n}$, and
$h_{i,n}:=h_{\alpha_i,n}$.

For a Lie algebra $\mathfrak a$ we denote its universal
enveloping algebra by $\U\mathfrak a$. For $k\in\C$ let
$\U_k\hat\fg$ be the quotient ring of $\U\hat\fg$ modulo the
ideal generated by $k-\mathbf{1}$. We call~$k$ \emph{the
level}. The representations of $\U_k\hat\fg$ are the same as
representations of $\hat\fg$ on which~$\mathbf{1}$ acts by $k$.
The critical level $k=-h^\vee$ is denoted by $k_c$.

To simplify notation we write $t^n\fu$ instead of $\fu\otimes
t^n\subset\hat\fg$, $t\fg[[t]]$ instead of
$\fg\otimes(t\C[[t]])$ etc.

For a scheme $Z$ we denote the set of functions on $Z$ by
$\Fun(Z)$. Thus if $Z$ is affine, then a quasi-coherent sheaf
on $Z$ is the same as a $\Fun(Z)$-module. We denote by $\Dmod(Z)$
the sheaf of differential operators on $Z$. If $Z$ is an affine
scheme, we shall abuse notation by denoting the ring of
differential operators by $\Dmod(Z)$ as well.

For a vector space $L$ we denote by $L^*$ the dual vector
space, and by $\Sym(L)$ its symmetric algebra. Thus
$\Sym(L)=\Fun(L^*)$.

\subsection{Non-highest weight modules}\label{NONHIGHEST}
Fix an integer $N\ge1$ and a level $k\ne k_c$. Let us introduce
the \emph{level subalgebra}
\[
    \fgp:=t^N(\fu_-\oplus\fu)\oplus t^{N+1}\fg[[t]]\subset
    \hat\fg.
\]
Let $\C_k$ be the $\fgp\oplus\C\mathbf{1}$-module on which
$\fgp$ acts by zero and $\mathbf{1}$ acts by $k$. Consider the
$\hat\fg$-module
\[
    \MM:=\Ind_{\fgp\oplus\C\mathbf{1}}^{\hat\fg}\C_k=
    \U_k\hat\fg\otimes_{\U\fgp}\C=
    \U_k\hat\fg/(\U_k\hat\fg\cdot\fgp).
\]
We denote the image of $1\in\U_k\hat\fg$ in $\MM$ by $\vac$ and
call it the \emph{vacuum vector}. Note that $\MM$ has a
universal property that
\[
    \Hom_{\U_k\hat\fg}(\MM,M)=M^\fgp.
\]
In particular, we have an identification of sets
$\End_{\hat\fg}(\MM)=(\MM)^\fgp$, given by $X\mapsto\vac\cdot
X$ (we view endomorphisms of left modules as acting on the
right). In particular we have
\begin{equation}\label{obvious}
    \U(t\fg[[t]]/\fgp)^{op}\subset\End_{\hat\fg}(\MM),
\end{equation}
where for a ring $A$ we denote by $A^{op}$ the same ring with
the opposite multiplication. It follows that we have a natural
inclusion $\Sym(t^N\fh)\hookrightarrow\End_{\hat\fg}(\MM)$. It
can be also described as a subring, generated by endomorphisms
$\vac\mapsto h_{i,N}\vac$.

Let $\MMR$ be the localization of $\MM$ given by
\begin{equation}\label{defmmr}
  \MMR:=\MM\bigotimes_{\Sym(t^N\fh)}\Fun(\fh^{*,r}),
\end{equation}
where we view $\Fun(\fh^{*,r})$ as a module over
$\Sym(t^N\fh)\approx\Fun(\fh^*)$.

\begin{remark}
Another way to define $\MMR$ is as follows: the action of
$\Sym(t^N\fh)\approx\Fun(\fh^*)$ makes $\MM$ a sheaf on
$\fh^*$, and $\MMR$ is the restriction of $\MM$ to~$\fh^{*,r}$.
\end{remark}

\begin{proposition}\label{PrEnd}
\begin{equation*}
    \End_{\hat\fg}(\MMR)=(\MMR)^\fgp.
\end{equation*}
\end{proposition}
\begin{proof}
A map $\End_{\hat\fg}(\MMR)\to(\MMR)^\fgp$ is given by
$X\mapsto(\vac\otimes1)\cdot X$. It follows from the PBW
theorem that if $\phi x=0$, where $x\in\MM$,
$\phi\in\Sym(t^N\fh)$, then $x=0$ or $\phi=0$. This implies
that the above map is injective.

It remains to show that every $x\in(\MMR)^\fgp$ gives rise to
an endomorphism of $\MMR$. This is somewhat similar to
the construction of the sheaf of differential operators on a
variety. First we check the following analogue of an Ore
condition: for all $x\in\MM$, $i=1,\ldots,\rk\fg$, and all
large enough $j$ we have $x\cdot(h_{i,N})^j\in h_{i,N}\MM$.
This, in turn, follows immediately from the fact that
$\ad_{h_{i,N}}$ is a nilpotent operator on $\MM$. (By
definition $\ad_{h_{i,N}}x:=h_{i,N}x-x\cdot h_{i,N}$.)

Thus for $x\in\MMR$ the equation $h_{i,N}y=x$ has a unique
solution. Therefore we can construct a \emph{left action\/} of
$\Fun(\fh^{*,r})$ on $\MMR$ extending the action of
$\Sym(t^N\fh)$. Now we can define the action of
$x\in(\MMR)^\fgp$ on $\MMR$ by $g\vac\otimes\phi\mapsto g(\phi
x)$, where $g\in\U_k\hat\fg$, $\phi\in\Fun(\fh^{*,r})$. We need
to check two things: (i) $g(\phi x)=0$ if $g\in\fgp$, and (ii)
$g\psi(\phi x)=g(\phi\psi x)$ whenever $\psi\in\Sym(t^N\fh)$.
For~(i), let us write $\phi=\phi_1/\phi_2$ with
$\phi_i\in\Sym(t^N\fh)$, and set $y:=(\phi_2)^{-1}x$. Since
$[t^N\fh,\fgp]\subset\fgp$, it is enough to show that $gy=0$ as
long as $g\in\fgp$. We have $g\phi_2^ky=0$ for all $k\ge1$.
Therefore $(\phi_2)^ngy=(\ad_{\phi_2}^ng)y$, and this is zero
if $n$ is large enough. Now (i) follows; (ii) is proved easily
by writing $\phi=\phi_1/\phi_2$ as before.
\end{proof}
\begin{corollary}
The multiplicative set generated by
$h_{i,N}\in\End_{\hat\fg}(\MM)$, where $i=1,\ldots,\rk\fg$,
satisfies the Ore conditions, and $\End_{\hat\fg}(\MMR)$ is the
localization with respect to this set. Also,
\begin{equation}\label{tensor}
    \End_{\hat\fg}(\MMR)=\End_{\hat\fg}(\MM)
    \bigotimes_{\Sym(t^N\fh)}\Fun(\fh^{*,r}),
\end{equation}
and $\End_{\hat\fg}(\MM)\subset\End_{\hat\fg}(\MMR)$.
\end{corollary}

\begin{remark}
We see that $\forall\:X\in\End_{\hat\fg}(\MMR)$ we have
$\prod_i h_{i,N}^{\ell_i}X\in\End_{\hat\fg}(\MM)$ for some
non-negative integers $\ell_i$.
\end{remark}

\begin{theorem}\label{ThEnd}
For $k\ne k_c$ there is a natural injective homomorphism of
rings
\begin{equation}\label{answer}
\textstyle\Dmod\left(\prod_1^{N-1}\fb^*\times\fh^{*,r}\right)\otimes\Fun(\fh^*)
\hookrightarrow\End_{\hat\fg}(\MMR).
\end{equation}
\end{theorem}
\begin{remarks}
1. We can re-write the ring of endomorphisms as
\[
    \textstyle\Dmod\left(\prod_1^{N-1}\fu^*\right)\otimes\Fun(\fh^*)
    \otimes\Dmod\left(\prod_1^{N-1}\fh^*\times\fh^{*,r}\right).
\]
In fact the endomorphisms, corresponding to the first multiple
as well as corresponding to
$\Fun\left(\prod_1^{N-1}\fh^*\times\fh^{*,r}\right)\subset
\Dmod\left(\prod_1^{N-1}\fh^*\times\fh^{*,r}\right)$ can
easily be seen: they correspond to the inclusion~(\ref{obvious})
(though the isomorphism between this subring and the LHS
of~(\ref{obvious}) is not obvious). The endomorphisms,
corresponding to the second multiple, are also easy to see.

The remaining endomorphisms are hidden. For example, as we
shall see below in case $N=1$, these endomorphisms correspond
to the Casimir connection operators.

2. We expect that the homomorphism is an isomorphism for all
$k\ne k_c$.

3. If $k=k_c$ the ring of endomorphisms is much bigger, one can
show that it contains the infinite-dimensional space of
functions on opers on a formal disc as a~subquotient.

4. The level group $\fgp$ looks a little bit unnatural. It
seems more reasonable to take $\fgp=t^N\fg[[t]]$ or
$\fgp=\fu[[t]]+t^N\fg[[t]]$. In the second case the results and
proofs are expected to be very similar, however, the
interpretation of $N=1$ case as the Casimir connection is not
clear.

The first case seems to be more complicated and we hope to
address it in future publications.

5. The above corollary shows that the inclusion~(\ref{obvious})
is strict and $\End_{\hat\fg}(\MM)$ has a lot of
``non-obvious'' endomorphisms.
\end{remarks}

\subsubsection{The right version}
There is a standard anti-involution $\iota:\U_k\hat\fg\to
\U_{-k}\hat\fg$ sending $x\in\hat\fg$ to $-x$. By composing the
actions of $\U_{-k}\hat\fg$ on $\mathbb{M}_{N,-k}$ or
$\mathbb{M}_{N,-k,reg}$ with $\iota$ we get right
$\U_k\hat\fg$-modules $\MM^r$ and $\MMR^r$. It is easy to check
that $\MM^r\approx\U_k\hat\fg/(\fgp\cdot \U_k\hat\fg)$ and
$\MMR^r\approx\Fun(\fh^{*,r})\otimes_{\Sym(t^N\fh)}\MM^r$. The
rings of endomorphisms are clearly the same, thus we have an
inclusion
\[
    \End_{\hat\fg}(\MMR^r)\supset
    \Dmod\left(\prod_1^{N-1}\fb^*\times
    \fh^{*,r}\right)\otimes\Fun(\fh^*).
\]

\subsection{The functor of coinvariants}
Let $M$ be any left $\hat\fg$-module of level $k$. Note that
$\Sym(t^N\fh)$ acts on the space of coinvariants $M/\fgp\cdot
M$ because $[t^N\fh,\fgp]\subset\fgp$. Thus we can form the
tensor product
\[
    \F(M):=\Fun(\fh^{*,r})
    \bigotimes_{\Sym(t^N\fh)}(M/\fgp\cdot M).
\]
\begin{lemma}
$\End_{\hat\fg}(\MMR^r)$ acts on $\F(M)$.
\end{lemma}
\begin{proof}
Clearly, $\End_{\hat\fg}(\MM^r)=(\U_k\hat\fg/(\fgp\cdot
\U_k\hat\fg))^\fgp$ acts on $M/\fgp\cdot M$. Thus its
localization $\End_{\hat\fg}(\MMR^r)$ acts on
\[
\End_{\hat\fg}(\MMR^r)\bigotimes_{\End_{\hat\fg}(\MM^r)}
(M/\fgp\cdot M)=\F(M).
\]
The last equality follows from the right version
of~({\ref{tensor}}).
\end{proof}

Thus we get a functor $\F$ from $\U_k\hat\fg$-mod to the
category of modules over
\[
    \textstyle\Dmod\left(\prod_1^{N-1}\fb^*\times\fh^{*,r}\right)
    \otimes\Fun(\fh^*).
\]
Such a module can be viewed as an $\fh^*$-family of
$\Dmod$-modules on $\prod_1^{N-1}\fb^*\times\fh^{*,r}$ and also
as a $\Dmod$-module on $\prod_1^{N-1}\fb^*\times\fh^{*,r}$ by
forgetting the action of $\Fun(\fh^*)$. In particular, it is a
quasi-coherent sheaf on $\prod_1^{N-1}\fb^*\times\fh^{*,r}$.

\subsection{Case $N=1$ and the Casimir connection}\label{CASIMIR} For $N=1$ we have
\[
  \F:\U_k\hat\fg\text{-mod}\to(\Dmod(\fh^{*,r})\otimes\Fun(\fh^*))\text{-mod}
  =(\Dmod(\fh^r)\otimes\Fun(\fh^*))\text{-mod},
\]
where we use the bilinear form to identify $\fh^*$ with $\fh$.
For any $\fg$-module $\V$ set
\[
    V:=\Ind_{\fg[t^{-1}]\oplus\C\mathbf{1}}^{\hat\fg}\V=
    \U_k\hat\fg\otimes_{\U(\fg[t^{-1}])}\V,
\]
where $t^{-1}\fg[t^{-1}]$ acts on $\V$ by zero, $\mathbf{1}$
acts by $k$.

Consider the trivial vector bundle $\V\times\fh^r$ over
$\fh^r$. Recall that Casi\-mir connection is a connection on
this bundle, see~\cite{ToledanoLaredo, ToledanoLaredo0,
ToledanoLaredo2}.

\begin{theorem}\label{ThCasimir}
\noindstep $\F(V)$ is naturally identified with $\V\times\fh^r$
as a sheaf on $\fh^r$.\\
\noindstep Under this identification the
$\Dmod(\fh^r)$-structure on $\F(V)$ is given by the Casi\-mir
connection up to a twist by a line bundle on $\fh^r\times\fh^*$
with a connection along~$\fh^r$.
\end{theorem}
\begin{remarks}
1. The structure of a sheaf on $\fh^*$ on $\F(V)$ comes from
the action of~$\fh$ on~$\V$, and this action commutes with both
the Casimir and our connections. Assume that~$\V$ is
finite-dimensional; then $\F(V)=\V\times\fh^r$ is supported at
finitely many points $\beta_i\in\fh^*$ as a sheaf on $\fh^*$
(the weight decomposition). The restrictions of two connections
to $\{\beta_i\}\times\fh^r$ differ by a line bundle with
connection. In particular the monodromies of our connection and
of the Casimir connection differ by an explicit scalar factor.

2. The first part of this theorem follows easily from PBW
theorem.

\end{remarks}
In concrete terms, the operators of our connection correspond
to the following endomorphisms of $\MMR^r$,
see~\S\ref{RIGHTCASIMIR}:
\[
    1\otimes\vac\mapsto1\otimes\vac h_{i,-1}+\sum_{\alpha\in\Delta_+}
 \alpha(h_i)h_{\alpha,1}^{-1}\otimes\vac f_{\alpha,0}e_{\alpha,0}.
\]
The corresponding ``regularized'' endomorphisms of $\MM$ are
given by
\begin{equation*}
    \vac\mapsto\vac h_{i,-1}\prod_{\beta\in\Delta_+}h_{\beta,1}+
    \sum_{\alpha\in\Delta_+}
    \alpha(h_i)\vac\prod_{\substack{\beta\in
    \Delta_+\\\beta\ne\alpha}} h_{\beta,1} f_{\alpha,0}e_{\alpha,0}.
\end{equation*}

\subsection{Irregular Wakimoto modules}
The idea of the proof of Theorem~\ref{ThEnd} is to construct an
isomorphism between $\MMR$ and another module, which we
call~\emph{an irregular Wakimoto module}. Its ring of
endomorphisms is easier to calculate. Irregular Wakimoto
modules will be defined by means of free field realization. Our
notation follows~\cite{FrenkelBook}.

Let $\mathcal A$ be the associative algebra with generators
$a_{\alpha,n}, a^*_{\alpha,n}$, $\alpha\in\Delta_+$, $n\in\Z$
and relations
\begin{equation*}
    [a_{\alpha,n},a^*_{\beta,m}]=
    \delta_{\alpha,\beta}\delta_{n,-m},\qquad
    [a_{\alpha,n},a_{\beta,m}]=
    [a^*_{\alpha,n},a^*_{\beta,m}]=0.
\end{equation*}
(The same algebra is defined in~\cite[\S5.3.3]{FrenkelBook},
where it is denoted by $\mathcal A^\fg$.)

Let $M_N$ be the module over $\mathcal A$ generated by a
vector, which we also denote by~$\vac$ and call a vacuum
vector, and relations:
\begin{equation}\label{vacuum}
a_{\alpha,n}\vac=0,\; n\ge N,\qquad
a^*_{\alpha,n}\vac=0,\; n\ge0.
\end{equation}
This module is analogous to the module $M_\fg$ defined
in~\cite[\S5.4.1]{FrenkelBook}. Note that~$a_{\alpha,n}$ with
$n\ge N$ and $a^*_{\alpha,n}$ with $n\ge0$ generate a
(commutative) subalgebra in $\mathcal A$.

Further, for $k\in\C$ let $\hat\fh_k$ be the Cartan part of
$\hat\fg_k$, where $\hat\fg_k$ is the central extension
$\hat\fg$, rescaled by $k$. More precisely, it is the Lie
algebra generated by~$b_{i,n}$, $i=1\ldots\rk\fg$, $n\in\Z$,
and a central element $\mathbf{1}$. The bracket is given by
\begin{equation}\label{binbjm}
    [b_{i,n},b_{j,m}]=-nk(h_i,h_j)\delta_{n,-m}\mathbf{1}.
\end{equation}

Let $\hat\pi_N^k$ be an $\hat\fh_k$-module generated by $\vac$
with relations
\[
b_{i,n}\vac=0,\;n>N,\quad\mathbf{1}\vac=\vac.
\]
Note the similarity between $\hat\pi_N^k$ and $\pi_0^k$
in~\cite[\S6.2.1]{FrenkelBook}. Let $\fhp$ be the subalgebra of
$\hat\fh^k$ generated by $b_{i,n}$ with $n>N$. If $\mathbf{1}$
acts on an $\hat\fh_k$-module $M$ by 1, then
\[
\Hom_{\hat\fh_k}(\hat\pi_N^k,M)=M^{\fhp}.
\]

Let $\pireg k$ be the localization of $\hat\pi_N^k$ defined as
follows: we have a unique endomorphism of $\hat\pi_N^k$ such
that $\vac\mapsto b_{i,N}\vac$ (follows from the universal
property of~$\hat\pi_N^k$). This gives an action of
$\Sym(t^N\fh)$ on $\hat\pi_N^k$. Analogously to (\ref{defmmr})
we set
\[
  \pireg k:=
  \hat\pi_N^k\bigotimes_{\Sym(t^N\fh)}\Fun(\fh^{*,r}).
\]
We call $M_N\otimes\pireg{k-k_c}$ an \emph{irregular Wakimoto
module} by analogy with $W_{\lambda,k}$ from~\cite[\S
6.2]{FrenkelBook}.

For a monomial $A$ in $a_{\alpha,n}$ and $a^*_{\alpha,m}$ we
define its normal ordering ${:}A{:}$ by moving all
$a_{\alpha,n}$ with $n\ge0$ and all $a^*_{\alpha,m}$ with $m>0$
to the right.

Define the generating functions
\begin{equation*}
\begin{split}
    a_\alpha(z):=&\sum_{n\in\Z} a_{\alpha,n}z^{-n-1},\quad
    a^*_\alpha(z):=\sum_{n\in\Z} a^*_{\alpha,n}z^{-n},\quad
    b_i(z):=\sum_{n\in\Z} b_{i,n}z^{-n-1},\\
    e_\alpha(z):=&\sum_{n\in\Z} e_{\alpha,n}z^{-n-1},\quad
    h_\alpha(z):=\sum_{n\in\Z} h_{\alpha,n}z^{-n-1},\quad
    f_\alpha(z):=\sum_{n\in\Z} f_{\alpha,n}z^{-n-1}.
\end{split}
\end{equation*}
Also, set $e_i(z):=e_{\alpha_i}(z)$, $f_i(z):=f_{\alpha_i}(z)$.

\begin{theorem}\label{ThFreeField}
\noindstep For certain polynomials $P_\beta^i$ and $Q_\beta^i$
without constant terms, and certain constants $c_i$ the
following formulae give a level $k$ action of $\hat\fg$ on
$M_N\otimes\hat\pi_N^{k-k_c}${\rm:}
\begin{equation}\label{ffr}
\begin{split}
e_i(z)\mapsto&\,a_{\alpha_i}(z)+\sum_{\beta\in\Delta_+}
{:}P_\beta^i(a_\alpha^*(z))a_\beta(z){:},\\
h_i(z)\mapsto&-\sum_{\beta\in\Delta_+}
\beta(h_i){:}a_\beta^*(z)a_\beta(z){:}+b_i(z),\\
f_i(z)\mapsto& \sum_{\beta\in\Delta_+}
{:}Q_\beta^i(a_\alpha^*(z))a_\beta(z){:}-
(c_i+(k-k_c)(e_i,f_i))\partial_za^*_{\alpha_i}(z)+
b_i(z)a^*_{\alpha_i}(z).
\end{split}
\end{equation}
\noindstep This action gives rise to an action of $\hat\fg$ on
$M_N\otimes\pireg{k-k_c}$.\\
\noindstep There is a unique isomorphism $\wp:\MMR\to
M_N\otimes\pireg{k-k_c}$ of $\hat\fg$-modules such that
$\vac\otimes1\mapsto\vac\otimes(\vac\otimes1)$. {\rm(}In the
RHS the first multiple is in $M_N$, the second multiple
$\vac\otimes1\in\pireg{k-k_c}$.{\rm)}
\end{theorem}
\begin{remark}
Our sign convention~(\ref{affinekacmoody}) does not agree
with~\cite[(1.3.3)]{FrenkelBook}. So our sign in front of the
second term in the last formula in~(\ref{ffr}) is different, as
well as the sign in~(\ref{binbjm}).
\end{remark}
In fact, part (a) is an easy consequence
of~\cite[Theorem~6.2.1]{FrenkelBook}, part~(b) is an easy
consequence of part~(a), the proof of part~(c) will occupy a
substantial part of this paper, it will be done by constructing
certain filtrations.

Theorem~\ref{ThEnd} follows from the above Theorem: the
required endomorphisms of $\MMR\approx M_N\otimes\pireg{k-k_c}$
are given by
$\vac\otimes(\vac\otimes1)\mapsto-a_{\alpha,n}\vac\otimes(\vac\otimes1)$,
where $0<n<N$, $\vac\otimes(\vac\otimes1)\mapsto
a^*_{\alpha,n}\vac\otimes(\vac\otimes1)$, where $-N<n<0$, and
$\vac\otimes(\vac\otimes1)\mapsto\vac\otimes(b_{i,n}\vac\otimes1)$,
with $-N\le n\le N$. We shall give more details
in~\S\ref{PrfThEnd}.

\subsection{The quasi-classical limit}\label{Isostokes}
Let $\hat\fg^*:=\fg^*\otimes\C((t))\,dt\oplus\C d$ be the
restricted dual to $\hat\fg$, where $d=\mathbf{1}^*$. Its
subset $\conn:=\fg^*\otimes\C((t))\,dt+d$ is identified
with the set of connections on the trivial $G$-bundle over the
punctured formal disc. Set
\[
    \conn_{N,\fh}:=\Bigl\{\nabla\in\conn:\;
    \nabla=d+\sum_{n\ge-N-1}A_nt^n\,dt,
    A_{-N-1}\in\fh^{*,r}\Bigl\}.
\]
We can view $\U_k\hat\fg$ and $\MMR$ as $\C[k]$-modules.
Choosing appropriate $\C[k^{-1}]$-lattices in
$\U_k\hat\fg\otimes\C[k,k^{-1}]$ and $\MMR\otimes\C[k,k^{-1}]$,
we extend $\U_k\hat\fg$ and $\MMR$ to $k=\infty$ (this is
similar to~\cite[\S16.3.2]{FrenkelBenZviBook}). Then one
identifies $\U_\infty\hat\fg$ with the algebra of functions on
$\conn$, and $\mathbb M_{N,\infty,reg}$ with the algebra of
functions on $\conn_{N,\fh}$. Note that $\conn$ is a Poisson
subspace of $\hat\fg^*$. Set $G_+:=\exp(\fgp)$; then
$\conn_{N,\fh}/G_+$ is an open subspace of the hamiltonian
reduction $\conn//G_+$. The following statement is very close
to~\cite[Proposition~5]{FedorovIsoStokes} and to standard
theorems about normal forms of connections.
\begin{lemma}\label{FormalNormalForm}
Every $G_+$-orbit in $\conn_{N,\fh}$ contains a unique element
of the form
\[
d+\sum_{n=-N-1}^{N-1}A_nt^ndt,
\]
where $A_{-N-1}\in\fh^{*,r}$, $A_n\in\fh$ for $n\ge-1$.
\end{lemma}
At $k=\infty$ the action of $\fgp$ on $\MMR$ becomes a
hamiltonian action of $\fgp$ on $\mathbb M_{N,\infty,reg}$.
Thus $(\mathbb M_{N,\infty,reg})^\fgp=\Fun(\conn_{N,\fh}/G_+)$
is the quasi-classical limit of
$\End_{\hat\fg}(\MMR)=(\MMR)^\fgp$. Thus the above lemma shows
that our (semi-conjectural) answer for $\End_{\hat\fg}(\MMR)$
has ``the right size''. One can show that the ``easy''
endomorphisms (see remark after Theorem~\ref{ThEnd}) correspond
to~$A_n$ with $n\le-1$.

Note also that the injection of Theorem~\ref{ThEnd} provides
Darboux coordinates on $\conn_{N,\fh}/G_+$.

\subsection{Global version}\label{GLOBAL}
Let $G((t))$ be the loop group of $G$, let $G_-\subset G$
denote the subgroup of loops that can be extended to a
$G$-valued function on $\Pro^1\setminus0$ that is equal to 1 at
infinity.

Set
\[
\Bun:=G_-\backslash G((t))/G_+.
\]
Then $\Bun$ is the moduli space of $G$-bundles on $\Pro^1$ with
the fiber at $\infty$ trivialized, and with a certain higher
level structure at zero. More precisely, it is a principal
$\fh$-bundle over the moduli space of $G$-bundles with the
fiber at $\infty$ trivialized and with the fiber at zero
trivialized to order $N-1$ (we view $\fh$ as an abelian Lie
group).

\begin{remark}
The space $\Bun$ looks a little bit unnatural. It would be more
reasonable to look at the moduli space of bundles trivialized
to order $N$ at zero, which would correspond to
$\fgp=t^{N+1}\fg[[t]]$, or having some higher order unipotent
structure, which would correspond to $\fgp=\fu[[t]]+
t^N\fg[[t]]$ (cf. the last remark after Theorem~\ref{ThEnd}).
\end{remark}

The determinant line bundle on $\Bun$ yields a 1-parametric
family of TDO, which we denote by $\Dmod_k(\Bun)$. It is easy
to see that each element of $\End_{\hat\fg}(\MM)$ gives rise to
a global section of $\Dmod_k(\Bun)$ (similar
to~\cite[Corollary~2.4.3]{FrenkelBenZvi}). We conjecture that
this construction gives all global level $k$ differential
operators on $\Bun$ for $k\ne k_c$. This conjecture is related
to a known fact that there are no non-constant differential
operators on the moduli space of bundles for $k\ne k_c$.

We expect similar statements to hold for other level groups.

\subsection{Recollection on isomonodromic deformation}
Let $\X\to S$ be a family of smooth projective curves,
$(\E,\nabla)$ be a family of bundles with connections.
Precisely,~$\E$ is a principal $G$-bundle on $\X$, $\nabla$ is
a~relative connection on $\E$ along the fibers of $\X\to S$.
This family is called \emph{isomonodromic}, if $\nabla$ can be
extended to an \emph{absolute flat\/} connection $\tilde\nabla$
on $\X$.

It is an easy exercise that this condition is equivalent to the
monodromy of~$\nabla$ being constant over $S$.

In the case of a connection with singularities the above
definition still makes sense. In the case of regular
singularities, the condition is still equivalent to the
monodromy being constant; in the case of irregular
singularities one has to require that both monodromy and the
so-called Stokes structures at the singularities do not change.
In this case a new direction of deformation arises: one can
deform \emph{formal normal forms\/} of connections as well. (We
assume the singularities to be generic.) We refer the reader
to~\cite{FedorovIsoStokes} for more details.

In the simplest case of connections on $\Pro^1$ with a~pole of
order 2 at zero, and a~pole of order 1 at infinity, the
isomonodromic deformation amounts to deforming the conjugacy
class of the leading term at zero. This isomonodromic
deformation is especially important because it is closely
related to Frobenius manifolds.

\subsection{Deformation to isomonodromy}\label{ISOMONODROMY}
The sheaf $\Dmod_k(\Bun)$ has a natural deformation to a
certain twisted cotangent bundle on $\Bun$. Precisely, this
twisted cotangent bundle is the moduli space of pairs
$(E,\nabla)$, where $E\in\Bun$, $\nabla$ is a connection on $E$
with a simple pole at $\infty$ such that $\nabla$ has the
following form at zero (in a trivialization compatible with the
level structure):
\[
d+dt\left(\frac h{t^{N+1}}+\mbox{higher order
terms}\right),\qquad
h\in\fh^*.
\]
Denote this moduli space by $\Conn$ (this is a moduli space of
extended connections), denote by $\ConnReg$ the open subspace
of $\Conn$ corresponding to $h\in\fh^{*,r}$. Trivializing a
bundle in the formal neighborhood of zero (this trivialization
is well defined up to the action of $G_+$), we get a natural
map $\ConnReg\to\conn_{N,\fh}/G_+$. The following statements
are very close to the results of~\cite{FedorovIsoStokes}:
\begin{enumerate}
\item $\ConnReg$ is a Poisson space.
\item The map $\ConnReg\to\conn_{N,\fh}/G_+$ is Poisson.
\item The pullbacks of functions along this map are
    hamiltonians of isomonodromic deformation (compare
    with~\cite[Theorem~2]{FedorovIsoStokes}).
\end{enumerate}
In~\S\ref{GLOBAL} we have constructed some global sections
of~$\Dmod_k(\Bun)$. Comparing with~\S\ref{Isostokes}, we see
that these sections are quantizations of isomonodromic
hamiltonians.

Note, that similar results are valid for any smooth projective
curve~$X$ and for any level structure divisor.

\section{Proofs of main results}
\subsection{Finite dimensional preliminaries} Recall first the
definition of polynomials~$P_\beta^i$ and $Q_\beta^i$
from~(\ref{ffr}). Recall that $U$ is identified with an open
subset of $G/B_-$, thus~$\fg$ acts on $U$. Also, $U$ is
isomorphic to $\fu$ as a variety. Fix a homogeneous coordinate
system $y_\alpha$, $\alpha\in\Delta_+$ on $U$. Here
\emph{homogeneous\/} means that
\begin{equation}\label{homo}
h\cdot y_\alpha=-\alpha(h)y_\alpha \forall\;h\in\fh,
\alpha\in\Delta_+.
\end{equation}
Then (possibly after multiplying coordinates $y_{\alpha_i}$ by
some constants) the action is given by
(cf.~\cite[\S5.2.5]{FrenkelBook})
\begin{equation}\label{findim}
\begin{split}
e_i\mapsto&\frac{\partial}{\partial
y_{\alpha_i}}+\sum_{\beta\in\Delta_+}
P_\beta^i(y_\alpha)\frac{\partial}{\partial y_\beta},\\
h_i\mapsto&-\sum_{\beta\in\Delta_+}
\beta(h_i)y_\beta\frac{\partial}{\partial y_\beta},\\
f_i\mapsto&\sum_{\beta\in\Delta_+}
Q_\beta^i(y_\alpha)\frac{\partial}{\partial y_\beta}.
\end{split}
\end{equation}
The Cartan subalgebra acts on $\fg$ and on $\Dmod(U)$
diagonally, so we can define the weight $w_\fh$ with respect to
the Cartan subalgebra. Clearly $w_\fh(e_\alpha)=\alpha$,
$w_\fh(h_i)=0$, $w_\fh(f_\alpha)=-\alpha$. Also,~(\ref{homo})
implies that $w(y_\alpha)=-\alpha$, $w(\partial/\partial
y_\alpha)=\alpha$. The action preserves $\fh$-weight. Later we
shall need the following
\begin{lemma}\label{LmEAlpha}
Multiplying the coordinates $y_\alpha$ by certain non-zero
constants, we may assume that for all $\alpha\in\Delta_+$
\begin{equation*}
e_\alpha\mapsto\frac{\partial}{\partial
y_\alpha}+\sum_{\substack{\beta\in\Delta_+\\ \beta>\alpha}}
P_\beta^\alpha(y_\gamma)\frac{\partial}{\partial y_\beta},
\end{equation*}
where $P_\beta^\alpha$ are certain polynomials without constant
terms.
\end{lemma}
\begin{proof}
We can write
\[
e_\alpha\mapsto
\sum_{\beta\in\Delta_+} c_\beta^\alpha\frac{\partial}{\partial
y_\beta}+\sum_{\beta\in\Delta_+}
P_\beta^\alpha(y_\gamma)\frac{\partial}{\partial y_\beta},
\]
where $c_\beta^\alpha$ are certain constants, $P_\beta^\alpha$
are polynomials without constant terms. The $\fh$-weight
considerations show that $c^\alpha_\beta$ is a diagonal matrix
and that $P_\beta^\alpha=0$ unless $\beta>\alpha$\,;
$c^\alpha_\beta$ is non-degenerate because the action of $\fu$
on $U$ is transitive, so in particular it is transitive at the
unit.
\end{proof}
\subsection{}
\emph{Proof of part {\rm(}a\/{\rm)} of
Theorem~\ref{ThFreeField}.} Let $V_k(\fg)$ be the level $k$
vacuum representation of $\hat\fg$ with its usual structure of
vertex algebra (see~\cite[\S2.2.4]{FrenkelBook}). Let $M_\fg$
be the Fock representation of $\mathcal A$,
see~\cite[\S5.4.1]{FrenkelBook}. Let $\pi_0^k$ be the vertex
algebra associated to $\hat\fh_k$. Then
by~\cite[Theorem~6.2.1]{FrenkelBook} there is a vertex algebra
homomorphism $V_k(\fg)\to M_\fg\otimes\pi_0^{k-k_c}$ given by
formulae~(\ref{ffr}).

It is clear that $M_N$ is a smooth representation of $\mathcal
A$, thus $M_N$ is a module over the vertex algebra $M_\fg$ with
the $M_\fg$-module structure induced from its $\mathcal
A$-module structure (see for
example~\cite[\S5.1.8]{FrenkelBenZviBook}). Similarly,
$\hat\pi_N^{k-k_c}$ is a module over the vertex algebra
$\pi_0^{k-k_c}$ with $\pi_0^{k-k_c}$-module structure induced
from $\hat\fh^{k-k_c}$-module structure. It follows
that~(\ref{ffr}) give a $V_k(\fg)$-module structure on
$M_N\otimes\hat\pi_N^{k-k_c}$. But this is equivalent to our
statement.\qed

\begin{proof}[Proof of part {\rm(}b{\rm)} of Theorem~\ref{ThFreeField}]
Since $\Sym(t^N\fh)$ acts on $\fh^{k-k_c}$-module
$\hat\pi_N^{k-k_c}$, it also acts on
$M_N\otimes\hat\pi_N^{k-k_c}$ by endomorphisms of $\mathcal
A\otimes\hat\fh^{k-k_c}$-module structure, thus it acts by
endomorphisms of the $\hat\fg$-module structure and the statement
follows.
\end{proof}

\subsection{Action of $\hat\fg$ revisited.}
Let us extend the $\fh$-weight to loop algebras, by setting
\[
\begin{split}
    w_\fh(e_{\alpha,n})&=w_\fh(a_{\alpha,n})=\alpha,\\
    w_\fh(f_{\alpha,n})&=w_\fh(a^*_{\alpha,n})=-\alpha,\\
    w_\fh(h_{i,n})&=w_\fh(b_{i,n})=0.
\end{split}
\]
Then~(\ref{ffr}) shows that the above action of $\hat\fg$ is
compatible with the $w_\fh$ gradation.

We would like to extend the formulae~(\ref{ffr}) to
$e_\alpha(z)$ and $f_\alpha(z)$, where $\alpha$ is not
necessarily a simple root. Let $b_\alpha(z)$ be the field
corresponding to $h_\alpha\in\fh$ (i.e. $b_\alpha(z)$ is a
linear combination of $b_i(z)$ with the same coefficients as in
the expression of $h_\alpha$ through $h_i$).

\begin{proposition}\label{PrAllRoots}
The action of $\hat\fg$ on $M_N\otimes\hat\pi_N^{k-k_c}$ is
given by
\begin{equation}\label{ffralpha}
\begin{split}
e_\alpha(z)\mapsto&\,a_{\alpha}(z)+\sum_{\substack{\beta\in\Delta_+\\
\beta>\alpha}}
{:}P_\beta^\alpha(a^*(z))a_\beta(z){:},\\
h_i(z)\mapsto&-\sum_{\beta\in\Delta_+}
\beta(h_i){:}a_\beta^*(z)a_\beta(z){:}+b_i(z),\\
f_\alpha(z)\mapsto&\sum_{\beta\in\Delta_+}
{:}Q_\beta^\alpha(a^*(z))a_\beta(z){:}+
\sum_{\beta\in\Delta_+}
\tilde Q_\beta^\alpha(a^*(z))\partial_z a^*_\beta(z)+\\
&b_\alpha(z)a^*_\alpha(z)+\sum_i
b_i(z)R_i^\alpha(a^*(z)),
\end{split}
\end{equation}
where\,{\rm:}
\begin{itemize}
  \item $P_\beta^\alpha$ are polynomials in $a^*_\gamma$
      without constant terms\smcn
  \item $Q^\alpha_\beta$ are polynomials in $a^*_\gamma$
      without constant terms\smcn
  \item $R_i^\alpha$ are polynomials in $a^*_\gamma$
      without constant and linear terms.
\end{itemize}
\end{proposition}
\begin{proof}
Recall from~\cite{FrenkelBook} how the homomorphism~(\ref{ffr})
is constructed. One starts with~(\ref{findim}) and extends it
to a map $\fg\otimes\C((t))\to\mathrm{Vect}(U((t)))$, where
$U((t))$ is the ind-scheme of loops in $U$,
\cite[\S5.3.2]{FrenkelBook}. The next step is to lift this map
to a map $V_{k_c}(\fg)\to M_\fg$, see~\cite{FrenkelBook},
Theorem~5.6.8 and Theorem~6.1.3.

It follows from Lemma~\ref{LmEAlpha} and discussion before
Theorem~6.1.3 in~\cite{FrenkelBook} that under this map
\begin{equation}\label{ealpha}
    e_\alpha(z)\mapsto\sum_{\beta\in\Delta_+}
    a_\alpha(z)+\sum_{\beta\in\Delta_+}
{:}P_\beta^\alpha(a^*(z))a_\beta(z){:}.
\end{equation}

The map $V_{k_c}(\fg)\to M_\fg$ can be deformed by any element
of $\fh^*((t))$, giving rise to a map $V_{k_c}(\fg)\to
M_\fg\otimes\pi_0$ but the images of $e_\alpha(z)$ are not
modified (see proof of~\cite[Lemma~6.1.4]{FrenkelBook}). The
last step is to deform the level to $k\ne k_c$. However, we
claim that the images of $e_{\alpha,n}$ do not depend on $k$.
Indeed, for $\alpha=\alpha_i$ it is clear from~(\ref{ffr}). All
other $e_{\alpha,n}$ can be obtained by commuting
$e_{\alpha_i,n}$. Finally, the fields~$b_i(z)$ are absent in
the RHS of~(\ref{ealpha}), so the result of commuting does not
depend on~$k$. This proves the first formula in the
proposition.

To calculate the images of $f_\alpha(z)$ we shall take a
different approach. First we show that
\[
f_\alpha(z)\mapsto\sum_{\beta\in\Delta_+}
{:}Q_\beta^\alpha(a^*(z))a_\beta(z){:}+
\sum_{\beta\in\Delta_+}
\tilde Q_\beta^\alpha(a^*(z))\partial_z a^*_\beta(z)+
\sum_i b_i(z)R'_i(a^*(z)).
\]
Indeed, the vertex algebra formalism shows that RHS has to be a
field, corresponding to some element of
$M_\fg\otimes\pi_0^{k-k_c}$ via the vertex operation. It also
has to be a~field of conformal dimension one. But it is easy to
see that this is the most general form of such a field. The
considerations of $\fh$-weight show that $R'_i$  and
$Q^\alpha_\beta$ have no constant terms. The same considerations
show that the linear term of $R'_i$ is of the form $\lambda_i
a_\alpha^*(z)$.

Next, we have
\[
    \Res_w[e_\alpha(z),f_\alpha(w)]dw=h_\alpha(z).
\]
The operator corresponding to LHS has the form
\[
\sum_i\lambda_i b_i(z)+\sum_i b_i(z)R''_i(a^*(z))+
\mbox{Terms without $b_i(z)$},
\]
where $R''_i$ have no constant terms. Comparing with the left
hand side we conclude that $\sum_i\lambda_i
b_i(z)=b_\alpha(z)$.
\end{proof}

\subsection{Constructing a map $\wp:\MMR\to
M_N\otimes\pireg{k-k_c}$} We want to introduce some notation.
We define the \emph{degree\/} of an element of $\mathcal
A\otimes\hat\fh^k$ by setting
\[
    \deg a_{\alpha,n}=\deg a^*_{\alpha,n}=
    \deg b_{i,n}=-n.
\]
Similarly, we define a degree of an element of $\U_k\hat\fg$.
Note that the commutation relations preserve the degree, so we
get gradations on $\mathcal A\otimes\hat\fh^k$ and
$\U_k\hat\fg$. A field $x(z)=\sum_n x_nz^{-n}$ is called
\emph{conformal\/} if $\deg x_n+n$ does not depend on $n$. For
a~conformal field $x(z)$ we define $x(z)_{(n)}$ to be a unique
$x_m$ with $\deg x_m=-n$. In particular
$a_\alpha(z)_{(n)}=a_{\alpha,n}$,
$a^*_\alpha(z)_{(n)}=a^*_{\alpha,n}$,
$e_\alpha(z)_{(n)}=e_{\alpha,n}$, etc.

For notational simplicity we denote $\vac\otimes\vac\in
M_N\otimes\hat\pi_N^k$ by $\vac'$. We also denote
$\vac\otimes(\vac\otimes1)\in M_N\otimes\pireg k$ by $\vac'$.
\begin{lemma}
$\vac'\in M_N\otimes\hat\pi_N^{k-k_c}$ is annihilated by
$\fgp$.
\end{lemma}
\begin{proof}
Let $A$ be the product of any number of generators
$a^*_{\alpha,n}$, with at least one $n>0$, then for any
$\gamma$ and $m$ we have ${:}Aa_{\gamma,m}{:}\vac'=0$ by
definition of normal ordering.

We show first that $e_{\alpha,n}\vac'=0$ if $n\ge N$. Indeed,
by Proposition~\ref{PrAllRoots} we have
\[
    e_{\alpha,n}\vac'=a_{\alpha,n}\vac'+
    \sum_{\beta\in\Delta_+}\sum_{\mu+\nu=n}
    {:}P_\beta^\alpha(a^*(z))_{(\mu)}a_{\beta,\nu}{:}\vac'.
\]
The first term is clearly zero; the terms with $\nu\ge N$ are
also zero. If $\nu<N$, then $\mu>0$, and the term is zero by
the remark in the beginning of the proof.

We leave it to the reader to check that $h_{i,n}\vac'=0$ for
$n>N$.

Finally, we have
\begin{multline*}
    f_{\alpha,n}\vac'=
\sum_{\beta\in\Delta_+}{:}(Q_\beta^\alpha(a^*(z))a_\beta(z))_{(n)}{:}\vac'+
\sum_{\beta\in\Delta_+} (\tilde
Q_\beta^\alpha(a^*(z))\partial_z a^*_\beta(z))_{(n)}\vac'\\+
\sum_{\mu+\nu=n} b_{\alpha,\mu}a^*_{\alpha,\nu}\vac'+
\sum_i\sum_{\mu+\nu=n}b_{i,\mu}R_i^\alpha(a^*(z))_{(\nu)}\vac'.
\end{multline*}
For the first two terms we note that every monomial has degree
$-n$. Thus, it has either a multiple $a_{\gamma,m}$ with $m\ge
N$ and the term is zero, or $a^*_{\gamma,m}$ with $m>0$, and
the term is again zero. The last two terms are clearly zero if
$\nu\ge0$. Otherwise, $\mu>N$, and we use the fact that $b$'s
and $a^*$'s commute.
\end{proof}
Thus there is a unique homomorphism $\MM\to
M_N\otimes\hat\pi_N^{k-k_c}$ sending $\vac$ to~$\vac'$.

\begin{lemma}\label{HinBin}
For all $i$ we have
\[
    h_{i,N}\vac'=b_{i,N}\vac'.
\]
\end{lemma}
\begin{proof}
By~(\ref{ffr})
\[
    h_{i,N}\vac'=-\sum_{\mu+\nu=N}\sum_{\beta\in\Delta_+}
\beta(h_i)a_{\beta,\mu}^*a_{\beta,\nu}\vac'+b_{i,N}\vac'
\]
(we can remove normal ordering because $a_{\beta,\mu}^*$ and
$a_{\beta,\nu}$ commute, since $\mu+\nu=N\ne0$). All terms
$a_{\beta,\mu}^*a_{\beta,\nu}\vac'$ are zero because $\mu\ge0$
or $\nu>N$.
\end{proof}
Thus our homomorphism intertwines the action of $t^N\fh$ on the
modules, so it gives rise to a homomorphism $\wp:\MMR\to
M_N\otimes\pireg{k-k_c}$.
\subsection{PBW bases and filtrations}
Choosing some order of positive roots, we get a basis of
\emph{$\Fun(\fh^{*,r})$-module\/} $\MMR$ given by
lexicographically ordered monomials
\begin{equation}\label{pbwbasis}
    J:=\prod_{\substack{n<N\\i}} (h_{i,n})^{\imath_{i,n}}
    \prod_{\substack{n<N\\ \alpha}} (f_{\alpha,n})^{\jmath_{\alpha,n}}
    \prod_{\substack{n<N\\ \beta}} (e_{\beta,n})^{\ell_{\beta,n}}\vac\otimes1,
\end{equation}
where all exponents $\imath_{i,n}$, $\jmath_{\alpha,n}$, and
$\ell_{\beta,n}$ are nonnegative integers. Also, we require
that all $h_{i,n}$ with $n>0$ are to the left from all
$h_{i,n}$ with $n<0$. Denote the set of such monomials by $\J$.
Every $x\in\MMR$ can be uniquely written as
\[
    \sum_{J\in\J}J\cdot\phi_J,
\]
where $\phi_J$ are functions on $\fh^{*,r}$.

Further, a basis of $\Fun(\fh^{*,r})$-module $M_N\otimes\pireg
k$ is given by monomials
\[
    A:=
    \prod_{\substack{n<N\\i}} (b_{i,n})^{\imath_{i,n}}
    \prod_{\substack{n<0\\ \alpha}} (a^*_{\alpha,n})^{\jmath_{\alpha,n}}
    \prod_{\substack{n<N\\ \beta}} (a_{\beta,n})^{\ell_{\beta,n}}\vac',
\]
where all exponents $\imath_{i,n}$, $\jmath_{\alpha,n}$, and
$\ell_{\beta,n}$ are nonnegative integers. We order the
multiples by a convention that all $b_{i,n}$ with $n>0$ are to
the left from all $b_{i,n}$ with $n<0$. We denote the set of
such monomials by $\B$.

Set
\[
I(h_{i,n})=I(f_{\alpha,n})=I(e_{\alpha,n})=N-n.
\]
Define \emph{the weight\/} $I(J)$ of a monomial $J\in\J$ by
extending the above assignment to monomials. Setting
$I(J\cdot\phi)=I(J)$ for $\phi\in\Fun(\fh^{*,r})$, we get a
filtration on $\MMR$. Note that $(\MMR)^{<0}=0$.

Similarly, by assigning
\[
    I(b_{i,n})=I(a_{\alpha,n})=N-n,\qquad
    I(a^*_{\alpha,n})=-n
\]
we get a filtration on $M_N\otimes\pireg k$, again
$(M_N\otimes\pireg k)^{<0}=0$.

To complete the proof of Theorem~\ref{ThFreeField}, we
shall~(i) note that the map $\wp$ preserves the filtrations;
(ii) identify $\gr\MMR$ and $\gr(M_N\otimes\pireg{k-k_c})$ with
rings of functions on certain schemes and (iii) identify the
induced morphism of the schemes.

\subsection{$\gr\U_k\hat\fg$ and $\gr\MMR$ as functions on
schemes} Recall that for a vector space $V$ we have the loop
space $V\otimes\C((t))$ and its restricted dual $V^*((t))\,dt$.
These spaces are ind-schemes. Further, $I$ gives a filtration
on $\U_k\hat\fg$ and we have
\[
    \gr\U_k\hat\fg=\Sym\fg((t))=\Fun(\fg^*((t))\,dt).
\]
(More precisely, $\fg^*((t))\,dt$ is an ind-scheme, so we
should view $\Sym\fg((t))$ as an algebra filtered by ideals.)

The filtration on $\U_k\hat\fg$ gives rise to a filtration on
its quotient module $\MM$. We see that $\gr\MM$ is a quotient
algebra of $\Sym\fg((t))$ equal to
\[
    \C[\bar h_{i,n},\bar e_{\alpha,m},\bar f_{\alpha,m}|\,n\le N,
    m<N],
\]
where we denote the image of $h_{i,n}$ in $\gr\MM$ by $\bar
h_{i,n}$ etc. Next,
\[
    \gr\MMR=\C[\bar h_{i,n},\bar h_{i,N}^{-1},\bar e_{\alpha,m},
    \bar f_{\alpha,m}|\,n\le N,m<N].
\]
It follows that $\gr\MMR$ can be identified with the ring of
functions on the scheme of infinite type
\[
    \tilde\fg^*:=\Biggl\{\,\sum_{n\ge-N-1}A_nt^n\,dt,
    \forall n\,A_n\in\fg^*,A_{-N-1}\in\fh^{*,r}\Biggr\}.
\]
The action of $\gr\U_k\hat\fg$ on $\gr\MMR$ is given by the
inclusion $\tilde\fg^*\hookrightarrow\fg^*((t))\,dt$.


\subsection{$\gr(\mathcal A\otimes\hat\fh_k)$ and $\gr(M_N\otimes\pireg k)$ as
functions on schemes} Set
\[
\begin{split}
   \tilde U:=&\{u\in U[[t]]:\;u(0)=1\}.\\
   \tilde\fb^*:=&\Biggl\{\,\sum_{n\ge-N-1}A_nt^n\,dt,
   \forall n\, A_n\in\fb^*, A_{-N-1}\in\fh^{*,r}\Biggr\}.\\
   \tilde\fu^*:=&\Biggl\{\,\sum_{n\ge-N}A_nt^n\,dt,
   \forall n\, A_n\in\fu^*\Biggr\}.\\
   \tilde\fh^*:=&\Biggl\{\,\sum_{n\ge-N-1}A_nt^n\,dt,
   \forall n\, A_n\in\fh^*, A_{-N-1}\in\fh^{*,r}\Biggr\}.
\end{split}
\]
Define
\[
    W:=\tilde U\times\tilde\fb^*=\tilde
    U\times\tilde\fu^*\times\tilde\fh^*.
\]
Note that all the spaces introduced are schemes (of infinite
type). Recall that $U((t))$ stands for the ind-scheme of loops
with values in $U$.
\begin{lemma}
There is a natural identification
\[
\begin{split}
\gr(\mathcal A\otimes\hat\fh_k)=&\Fun\bigl(U((t))
\times(\fb^*((t))\,dt)\bigr),\\
\gr(M_N\otimes\pireg k)=&\Fun(W),
\end{split}
\]
so that the action of $\gr(\mathcal A\otimes\hat\fh_k)$ on
$\gr(M_N\otimes\pireg k)$ is given by the natural inclusion
$W\hookrightarrow U((t))\times(\fb^*((t))\,dt)$.
\end{lemma}
\begin{proof}
Let us identify $T^*U=U\times\fu^*$ using the left
trivialization. Let $w_\alpha$ be the vector field
$\partial/\partial y_\alpha$, viewed as a function on $T^*U$.
Then $(y_\alpha,w_\alpha)$ is a system of coordinates on
$U\times\fu^*$. Let $y_{\alpha,n}$ and $w_{\alpha,n}$ be the
corresponding jet coordinates on
$U((t))\times(\fu^*((t))\,dt)$. That is, viewing
$(u(t),A(t)\,dt)\in U((t))\times(\fu^*((t))\,dt)$ as
a~$\C((t))$-valued point of $T^*U$, we have:
\[
    \sum y_{\alpha,n}t^n=y_\alpha(u(t)),\qquad
    \sum w_{\alpha,n}t^n=w_\alpha(u(t),A(t)).
\]
Identifying $\bar a_{\alpha,n}\mapsto w_{\alpha,-n-1}$, $\bar
a^*_{\alpha,n}\mapsto y_{\alpha,-n}$ we obtain
\[
    \gr\mathcal
    A=\Fun(U((t))\times(\fu^*((t))\,dt)).
\]
Further, $\gr\hat\fh_k=\C[\bar b_{i,n}]=\Fun(\fh^*((t))\,dt)$
and the first statement of the lemma follows.

Clearly $\{y_{\alpha,n},w_{\alpha,n}|\,n\ge0\}$ is a coordinate
system on the jet scheme $U[[t]]\times(\fu^*[[t]]\,dt)$. Note
that $w_\alpha$ is linear on the fibers of $T^*U\to U$, so we
have
\[
    w_\alpha(u(t),t^{-N}A(t))=t^{-N}w_\alpha(u(t),A(t)),
\]
therefore
\[
\{y_{\alpha,n},\,n\ge0,\quad w_{\alpha,m},\,m\ge-N\}
\]
is a coordinate system on $U[[t]]\times\tilde\fu^*$. Note that
$u(t)\in U[[t]]$ satisfies $u(0)=1$ iff $y_{\alpha,0}=0$ for
all $\alpha$. Hence $\gr M_N=\Fun(\tilde U\times\tilde\fu^*)$.
Now the second statement of the lemma follows. The rest of the
lemma is now obvious.
\end{proof}

\subsection{End of proof of Theorem~\ref{ThFreeField}}
In part~(b) of Theorem~\ref{ThFreeField} we have constructed an
action of $\hat\fg$ on $M_N\otimes\pireg{k-k_c}$. It gives rise
to an action of $\gr\U_k\hat\fg$ on
$\gr(M_N\otimes\pireg{k-k_c})$, which actually comes from a
homomorphism of algebras.

\begin{proposition}\label{ClassicalFree}
The action of $\gr\U_k\hat\fg$ on
$\gr(M_N\otimes\pireg{k-k_c})$ is induced by
$\Phi:W\to\fg^*((t))\,dt$ given by
\[
    \Phi(u,A)=\Ad_u A,\quad
    u\in\tilde U,A\in\tilde\fb^*.
\]
\end{proposition}
\begin{proof}
\emph{Step 1.} From the proof of part {\rm(}a\/{\rm)} of
Theorem~\ref{ThFreeField}, recall the vertex algebras
$V_k(\fg)$ and $M_\fg\otimes\pi_0^{k-k_c}$. The weight $I$
defined above gives a filtration on these algebras as well, and
we claim that $\gr
V_k(\fg)=\Fun(\fg^*[[t]]\,dt)=\Fun(\fg[[t]])$, where we
identify $\fg$ and $\fg^*$ via a non-degenerate pairing. This
statement would follow immediately from the proof of
Proposition~7.1.1 in~\cite{FrenkelBook} if we take the usual
Poincare--Birkhoff--Witt filtration. However, $V_k(\fg)$ has a
grading by degree, and on $j$-th graded piece we have
$I(J)=N\cdot PBW(J)-j$, where we denote by $PBW(J)$ the
$PBW$-weight of a monomial. It follows that the associated
graded space is the same for both filtrations.

Similarly,
\[
    \gr(M_\fg\otimes\pi_0^{k-k_c})=\Fun(U[[t]]\times(\fb^*[[t]]\,dt))
    =\Fun(U[[t]]\times\fb^-[[t]]).
\]
Finally, the map on associated graded vertex algebras, induced
from the free field realization map $V_k(\fg)\to
M_\fg\otimes\pi_0^{k-k_c}$ is induced from the map
$U[[t]]\times(\fb^*[[t]]\,dt)\to\fg^*[[t]]$ given by
$(u,A)\mapsto\Ad_uA$. (cf.~\cite[(7.1.5)]{FrenkelBook}.)

\emph{Step 2.} The map of vertex algebras induces a map of
their enveloping algebras
\[
    \Sym\fg((t))\to\gr(\mathcal A\otimes\hat\fh_{k-k_c}).
\]
It is easy to see that it is induced by the map $U((t))\times
(\fb^*((t))\,dt)\to\fg^*((t))\,dt$ given by the same formula
$(u,A)\mapsto\Ad_uA$.

Since the action of $\gr(\mathcal A\otimes\hat\fh_{k-k_c})$ on
$\gr(M_N\otimes\pireg{k-k_c})$ is given by the natural
inclusion $W\hookrightarrow U((t))\times(\fb^*((t))\,dt)$, the
proposition follows.
\end{proof}

\begin{corollary}
\noindstep The map
$\gr\wp:\gr\MMR\to\gr(M_{N,k}\otimes\pireg{k-k_c})$ is induced
by $\Phi':W\to\tilde\fg^*$ given by
\[
    \Phi'(u,A)=\Ad_uA.
\]
\noindstep $\Phi'$ is an isomorphism.
\end{corollary}
\begin{proof}
(a) Follows obviously from the proposition.\\
(b) We need to check that every $A\in\tilde\fg^*$ can be
conjugated to an element of~$\tilde\fb^*$ by a unique element
of $\tilde U$. We look for such an element of $\tilde U$ in the
form $\exp(u_1t+u_2t^2+\ldots)$. We can find $u_i$'s one-by-one,
using the fact that the adjoint action of any element of
$\fh^r$ on $\fu$ is invertible,
cf.~\cite[Proposition~5]{FedorovIsoStokes}.
\end{proof}

We see that $\gr\wp$ is an isomorphism. So $\wp$ is an
isomorphism. This completes the proof of
Theorem~\ref{ThFreeField}.

\subsection{Proof of Theorem~\ref{ThEnd}}\label{PrfThEnd}
Let $\B_0\subset\B$ be the set of PBW monomials, which are
products of $a_{\alpha,n}$, with $0<n<N$, $a^*_{\alpha,n}$,
where $-N<n<0$ and $b_{i,n}$, with $-N\le n<N$. Set
$\U_1\hat\fh_k:=\U\hat\fh_k/\U\hat\fh_k\cdot(1-\mathbf{1})$.
Let $(\mathcal A\otimes\U_1\hat\fh_k)_+$ be the left ideal in
$\mathcal A\otimes\U_1\hat\fh_k$ generated by
$a_{\alpha,n}\otimes1$ with $n\ge N$, $a^*_{\alpha,n}\otimes1$
with $n\ge0$ and $1\otimes b_{i,n}$ with $n>N$.

It is easy to see that every $A\in\B_0$ is annihilated by
$(\mathcal A\otimes\U_1\hat\fh_{k-k_c})_+$. Let $\Lambda$ be
the $\Fun(\fh^{*,r})$-submodule of $M_N\otimes\pireg{k-k_c}$
generated by $\B_0$, it is easy to see that
\[
    \Lambda=(M_N\otimes\pireg{k-k_c})^{(\mathcal
    A\otimes\U_1\hat\fh_{k-k_c})_+}=
    \End_{\mathcal
    A\otimes\U_1\hat\fh_{k-k_c}}(M_N\otimes\pireg{k-k_c}).
\]
The last identity is proved in the same way as
Proposition~\ref{PrEnd}.

Thus $\Lambda\subset\End_{\hat\fg}(M_N\otimes\pireg{k-k_c})$.
Let us calculate the ring structure on $\Lambda$. Viewing
$a_{\alpha,n},a^*_{\beta,m}\in\B_0$ as endomorphisms and
recalling that by our convention they act on the right, we get
\begin{multline*}
    \vac'{[a_{\alpha,n},a^*_{\beta,m}]}_\Lambda=
    (a^*_{\beta,m}\vac')a_{\alpha,n}-
    (a_{\alpha,n}\vac')a^*_{\beta,m}=\\
    a^*_{\beta,m}a_{\alpha,n}\vac'-
    a_{\alpha,n}a^*_{\beta,m}\vac'=
    -\delta_{\alpha,\beta}\delta_{n,-m}\vac',
\end{multline*}
where ${[\cdot,\cdot]}_\Lambda$ is the commutator in the ring
$\Lambda$. Thus
\[
{[a_{\alpha,n},a^*_{\beta,m}]}_\Lambda=-\delta_{\alpha,\beta}\delta_{n,-m}.
\]
Similarly,
\begin{equation}\label{commrel2}
    {[b_{i,n},b_{j,m}]}_\Lambda=n(k-k_c)(h_i,h_j)\delta_{n,-m},
\end{equation}
and all the other commutators are zero. We want to identify
$\Lambda$ with the LHS of~(\ref{answer}). To this end we
identify $a^*_{\alpha,-n}$ with a coordinate system on $n$-th
copy of~$\fu^*$, we identify $a_{\alpha,n}$ with
$-\partial/\partial a^*_{\alpha,-n}$.

Next, for $0<n<N$, we identify $b_{i,n}$ with the function
$h_i$ on $n$-th copy of $\fh^*$, we identify $b_{i,N}$ with the
function $h_i$ on $\fh^{*,r}$, hence $b_{i,N}$ can occur in
negative powers. We want to identify $b_{i,-n}$ with a
differential operator on $n$-th copy of~$\fh^*$ corresponding
to a constant vector field. The commutation
relations~(\ref{commrel2}) show that
\begin{equation*}
    b_{i,-n}\mapsto
    -\frac{2n(k-k_c)}{(\alpha_i,\alpha_i)}\partial_{\alpha_i}.
\end{equation*}
Finally we identify $b_{i,0}$ with the function $h_i$ on the
``extra'' copy of $\fh^*$. This gives the required isomorphism.

\subsection{Case $N=1$}
According to the proof of Theorem~\ref{ThEnd}, the ring
$\Lambda$ in the case $N=1$ is generated by
$\vac\otimes1\mapsto\wp^{-1}(b_{i,-1}\vac')$,
$\vac\otimes1\mapsto\wp^{-1}(b_{i,0}\vac')$,
and $\vac\otimes1\mapsto\wp^{-1}(b_{i,1}^{\pm1}\vac')$, where
$i=1,\ldots,\rk\fg$. Let us calculate these $\wp$-preimages.
Let $\rho$ be the half-sum of positive roots.

\begin{proposition}
\noindstep $\wp(h_{i,1}\vac\otimes1)=b_{i,1}\vac'$\smcn\\
\noindstep $\wp((h_{i,0}-2\rho(h_i))\vac\otimes1)=b_{i,0}\vac'$\smcn\\
\noindstep $\wp(h_{i,-1}\vac\otimes1+\sum_{\alpha\in\Delta_+}
    \alpha(h_i)e_{\alpha,0}f_{\alpha,0}\vac\otimes
    h_{\alpha,1}^{-1})=b_{i,-1}\vac'$.
\end{proposition}
\begin{proof}
Part~(a) follows from Lemma~\ref{HinBin}. For part~(b) we have
\[
h_{i,0}\vac'=-\sum_{\alpha\in\Delta_+}
\alpha(h_i){:}a^*_{\alpha,0}a_{\alpha,0}{:}\vac'+b_{i,0}\vac'.
\]
But by our definition of normal ordering we have:
\[
{:}a^*_{\alpha,0}a_{\alpha,0}{:}\vac'=
a^*_{\alpha,0}a_{\alpha,0}\vac'=
-\vac'+a_{\alpha,0}a^*_{\alpha,0}\vac'=-\vac',
\]
and the statement follows.

Let us prove part~(c). We have
\begin{equation}\label{start}
    h_{i,-1}\vac'=-\sum_{\alpha\in\Delta_+}
    \alpha(h_i)a_{\alpha,-1}^*a_{\alpha,0}\vac'+b_{i,-1}\vac'.
\end{equation}
Further,
\begin{multline}\label{terms}
     f_{\alpha,0}\vac'=\sum_{\beta\in\Delta_+}
{:}(Q_\beta^\alpha(a^*(z))a_\beta(z))_{(0)}{:}\vac'+
\sum_{\beta\in\Delta_+}
(\tilde Q_\beta^\alpha(a^*(z))\partial_z a^*_\beta(z))_{(0)}\vac'+\\
     (b_\alpha(z)a^*_\alpha(z))_{(0)}\vac'
     +\sum_i (b_i(z)R_i^\alpha(a^*(z)))_{(0)}\vac'.
\end{multline}
The second terms vanishes: indeed, for a monomial $A$ in
$a^*_{\gamma,\mu}$ of degree zero it is clear that $A\vac'=0$.

Let us show that the first term vanishes. Consider the
expression ${:}Aa_{\beta,\nu}{:}\vac'$, where $A$ is a monomial
in $a^*_{\gamma,\mu}$ of degree $\nu$. If $\nu>0$, we clearly
get zero. If~$\nu<0$, then $A$ contains $a^*_{\gamma,\mu}$ with
$\mu>0$ and the term vanishes. If $\nu=0$, and~$A$ contains
$a^*_{\gamma,\mu}$ with $\mu>0$, we again get zero. Thus the
only interesting case is when $A=\prod a^*_{\gamma,0}$. If
there are at least two multiples in $A$ we still get zero. In
the remaining case:
\[
    {:}a^*_{\gamma,0}a_{\beta,0}{:}\vac'=
    -\delta_{\gamma,\beta}\vac'=0,
\]
since $\gamma=\beta$ is impossible due to $\fh$-weight
considerations.

The term $b_{i,\mu}R_i^\alpha(a^*(z))_{(-\mu)}$ vanishes if
$\mu>1$ and if $\mu\le 0$. It also vanishes when $\mu=1$
because any monomial in $R_i^\alpha(a^*(z))_{(-1)}$ has at
least two terms, and one of them has to kill $\vac'$. The term
$b_{\alpha,\mu}a^*_{\alpha,-\mu}$ vanishes unless $\mu=1$. Thus
$f_{\alpha,0}\vac'=b_{\alpha,1}a^*_{\alpha,-1}\vac'$ and
$\wp(f_{\alpha,0}\vac\otimes
h_{\alpha,1}^{-1})=a^*_{\alpha,-1}\vac'$.

Let us calculate
\begin{multline}\label{}
    \wp(e_{\alpha,0}f_{\alpha,0}\vac\otimes
    h_{\alpha,1}^{-1})=e_{\alpha,0}a^*_{\alpha,-1}\vac'=\\
    a_{\alpha,0}a^*_{\alpha,-1}\vac'+\sum_{\substack{\beta\in\Delta_+\\ \beta>\alpha}}
    \sum_\nu{:}P_\beta^\alpha(a^*(z))_{(-\nu)}a_{\beta,\nu}{:}a^*_{\alpha,-1}\vac'.
\end{multline}
For $\nu\ne1$, the last term vanishes for the same reasons as
the first term in~(\ref{terms}). For $\nu=1$ it vanishes
because $\beta>\alpha$, so that $a_{\beta,1}$ and
$a^*_{\alpha,-1}$ commute.

Thus $\wp(e_{\alpha,0}f_{\alpha,0}\vac\otimes
h_{\alpha,1}^{-1})=a_{\alpha,0}a^*_{\alpha,-1}\vac'$. Combining
with~(\ref{start}) we get the proposition.
\end{proof}

\subsubsection{Right version}\label{RIGHTCASIMIR} Applying the involution $\iota$
we see that $\Lambda^r\subset\End_{\hat\fg}(\mathbb
M_{1,k,reg}^r)$ is generated by
\[
\begin{split}
\hat b_{i,1}:1\otimes\vac\mapsto&1\otimes\vac h_{i,1},\\
\hat b_{i,0}:1\otimes\vac\mapsto&1\otimes\vac h_{i,0},\\
\hat b_{i,-1}:1\otimes\vac\mapsto&1\otimes\vac h_{i,-1}+\sum_{\alpha\in\Delta_+}
 \alpha(h_i)h_{\alpha,1}^{-1}\otimes\vac f_{\alpha,0}e_{\alpha,0}.
\end{split}
\]
The ring structure is given by
\[
[\hat b_{i,-1},\hat b_{j,1}]=(-k-k_c)(h_i,h_j),
\]
the other commutators are zero. The isomorphism
$\End_{\hat\fg}(\mathbb
M_{1,k,reg}^r)\approx\Dmod(\fh^{*,r})\otimes\Fun(\fh^*)$ is
given by $\hat b_{i,1}\mapsto h_i$, $\hat b_{i,-1}\mapsto-
\frac{2(k+k_c)}{(\alpha_i,\alpha_i)}\partial_{\alpha_i}$.

Identifying $\fh^*$ with $\fh$ via the bilinear form, we get
\begin{equation}\label{conn}
\hat b_{i,1}\mapsto2\alpha_i/(\alpha_i,\alpha_i),\qquad \hat
b_{i,-1}\mapsto(-k-k_c)\partial_{h_i}.
\end{equation}

\subsection{Proof of Theorem~\ref{ThCasimir}}
By PBW Theorem the space of coinvariants $V/\fgp\cdot V$ is
naturally isomorphic to
$\Sym(t\fh)\otimes\V=\Fun(\fh)\otimes\V$, where we use the
bilinear form to identify $t\fh$ and $\fh^*$. Note that under
this identification $h_{\alpha,1}$ is identified with
$2\alpha/(\alpha,\alpha)$. Localizing we get
$\F(V)=\Fun(\fh^r)\otimes\V$.

Denote the connection given by the $\Dmod$-module structure on
$\F(V)$ by $\nabla$. Set $\hbar:=\frac{-1}{2(k+k_c)}$.
By~(\ref{conn}) we have
\[
    \nabla_{h_i}(\phi\otimes v)=2\hbar\hat b_{i,-1}(\phi\otimes v),
\]
where $\phi\otimes v\in\Fun(\fh^r)\otimes\V$. To calculate the
connection it is enough to calculate
\[
\hat b_{i,-1}(1\otimes v)=
h_{i,-1}(1\otimes v)+\sum_{\alpha\in\Delta_+}
\alpha(h_i)\frac{(\alpha,\alpha)}{2\alpha}\otimes
f_{\alpha,0}e_{\alpha,0}v.
\]
Note that $h_{i,-1}v=0$. Using the Leibnitz rule we get
\[
\nabla_{h_i}(\phi\otimes v)=\partial_{h_i}\phi\otimes v+
2\hbar\sum_{\alpha\in\Delta_+}\frac{\alpha(h_i)}{\alpha}\phi
\otimes\frac{(\alpha,\alpha)}2f_\alpha e_\alpha v.
\]
Switching to differential forms we re-write:
\[
\nabla=d+2\hbar\sum_{\alpha\in\Delta_+}
\frac{d\alpha}{\alpha}
\left(\frac{(\alpha,\alpha)}2f_\alpha e_\alpha\right).
\]
Let us compare it with the Casimir
connection~\cite[\S2]{ToledanoLaredo0}
\[
\nabla_{Casimir}=d+\hbar\sum_{\alpha\in\Delta_+}\frac{d\alpha}{\alpha}
\left(\frac{(\alpha,\alpha)}2(e_\alpha f_\alpha+f_\alpha e_\alpha)\right)=
\nabla+\hbar\sum_{\alpha\in\Delta_+}\frac{d\alpha}{\alpha}
\frac{(\alpha,\alpha)}2h_\alpha.
\]
Now recall that $\F(V)$ is a sheaf on $\fh^r\times\fh^*$ with a
connection along $\fh^r$. Let $\cL$ be the trivial line bundle
on $\fh^r\times\fh^*$ with a connection along $\fh^r$ given by
\[
    \nabla_0=d+\hbar\sum_{\alpha\in\Delta_+}\frac{d\alpha}{\alpha}
    \frac{(\alpha,\alpha)}2h_\alpha,
\]
where $h_\alpha$ is now viewed as a function on $\fh^*$. We see
that
\[
    (\F(V),\nabla_{Casimir})=(\F(V),\nabla)\otimes(\cL,\nabla_0).
\]
and the Theorem is proved.
\begin{remark}
We have used a version of the Casimir connection with truncated
$\fsl_2$ Casimir operators $\frac{(\alpha,\alpha)}2(e_\alpha
f_\alpha+f_\alpha e_\alpha)$. If we use the usual Casimir
operators $\frac{(\alpha,\alpha)}2(e_\alpha f_\alpha+f_\alpha
e_\alpha+\frac12h_\alpha^2)$ (cf.~\cite[\S5]{ToledanoLaredo}
and~\cite{ToledanoLaredo2}) instead, then we get similar
results with
\[
    \nabla_0=d+\hbar\sum_{\alpha\in\Delta_+}\frac{d\alpha}{\alpha}
    \frac{(\alpha,\alpha)}2\left(h_\alpha+\frac12h_\alpha^2\right).
\]
\end{remark}

\bibliographystyle{alphanum}
\bibliography{Casimir}

\end{document}